\documentclass[12pt]{article}
\usepackage{amsmath,amssymb,graphicx}
\newtheorem{prethm}{{\bf Theorem}}

\newenvironment{thm}{\begin{prethm}{\hspace{-0.5
               em}{\bf.}}}{\end{prethm}}

\newtheorem{prepro}[prethm]{Proposition}

\newtheorem{prelem}[prethm]{Lemma}

\newtheorem{precor}[prethm]{Corollary}
\newenvironment{cor}{\begin{precor}{\hspace{-0.5
               em}{\bf.}}}{\end{precor}}

\newtheorem{prerem}[prethm]{{\bf Remark}}
\newenvironment{rem}{\begin{prerem}\em{\hspace{-0.5
              em}{\bf.}}}{\end{prerem}}

\newtheorem{preexample}{{\bf Example}}

\newtheorem{preproof}{{\bf Proof.}}

\newenvironment{proof}[1]{\begin{preproof}{\rm
               #1}\hfill{$\Box$}}{\end{preproof}}

\newcommand{\M}{{\cal M}}

\newcommand{\x}{{\bf x}}
\newcommand{\y}{{\bf y}}
\newcommand{\0}{{\bf 0}}

\renewcommand{\thefootnote}

\title{Inverses of triangular matrices and bipartite graphs}
\author{ R.B. Bapat$^{\,\rm a}$ \quad  E. Ghorbani$^{\,\rm b,c}$ \\ [.3cm]
{\small $^{\rm a}$Indian Statistical Institute, Delhi Centre, 7 S.J.S.S. Marg,}\\ {\small New Delhi 110 016, India}\\
{\small $^{\rm b}$Department of Mathematics, K.N. Toosi University of Technology,}\\
{\small P. O. Box 16315-1618, Tehran, Iran}\\
{\small $^{\rm c}$School of Mathematics, Institute for Research in Fundamental
Sciences (IPM),}\\{\small P.O. Box
19395-5746, Tehran, Iran }
\\[.3cm]{
$\mathsf{rbb@isid.ac.in}$ \quad  $\mathsf{e\_ghorbani@ipm.ir}$ }}

\date{}

\begin{document}
\maketitle

\begin{abstract}
To a given nonsingular triangular matrix $A$
 with entries from a ring, we associate a weighted
 bipartite graph $G(A)$ and give a combinatorial description of the inverse of $A$ by
 employing paths in  $G(A)$.
Under a certain condition, nonsingular triangular
matrices $A$ such that $A$ and $A^{-1}$ have the same zero-nonzero pattern are characterized.
 A combinatorial construction is given to construct outer inverses of the adjacency matrix
of a weighted tree.

\vspace{3mm}
\end{abstract}

\section{Introduction}

Let $A$ be a lower triangular matrix with entries from a ring,
which is not necessarily commutative. In the first section of this paper we obtain a combinatorial
formula for $A^{-1},$ when it exists. The formula is in terms of
certain paths in the bipartite graph associated with $A.$ We note
some consequences of this formula which include expressions
for the inverse of a block triangular matrix and a formula for the
inverse of the adjacency matrix of a bipartite graph with a
unique perfect matching.

In Section 3 we consider lower triangular, invertible, nonnegative
matrices $A$ and characterize those such that $A$ and $A^{-1}$
have the same zero-nonzero pattern. This relates to a question posed by Godsil \cite{g} for bipartite graphs.
In the final section we provide a combinatorial construction
of outer inverses of the adjacency matrix of a weighted tree.

\section{Inverses of triangular matrices}

Let $G$ be a bipartite graph and let ${\cal M}$ be  a matching in $G.$
We assume that each edge $e$ of $G$ has a nonzero weight $w(e)$ from a ring (not necessarily commutative).
A path in $G$ is said to be {\em alternating} if the edges are alternately
in ${\cal M}$ and ${\cal M}^c,$ with the first and the last edges being in ${\cal M}.$
A path with only one edge, the edge being in ${\cal M},$
is alternating.
Let $P$ be the alternating path consisting of the edges $e_1,e_2, \ldots,e_k$ in that order.
The {\em weight} $w(P)$ of $P$ is defined to be
$w(e_1)^{-1}w(e_2)w(e_3)^{-1}\cdots w(e_{k-1})w(e_k)^{-1},$
assuming that the inverses exist. Thus, if the weights
commute, then $w(P)$ is just the
product of the weights of the edges in $P \cap {\cal M}^c$ divided by
the product of the weights of the edges in $P \cap {\cal M}.$
The length $\ell(P)$ of $P$ is the number of edges
on that.
For an alternating path $P,$ we define $$\epsilon (P) = (-1)^{
(\ell(P)-1)/2}.$$

Let $A$ be an $n \times n$ matrix with entries from a ring. We associate a bipartite
graph $G(A)$ with $A$ as usual: the vertex set is $\{R_1, \ldots, R_n\} \cup \{
C_1, \ldots, C_n\}$ and there is an edge $e$ between $R_i$ to $C_j$ if and only
if $a_{ij} \ne0,$ in which case we assign $e$ the weight $w(e)=a_{ij}$.
We write vectors as row vectors. The transpose of $\x$ is denoted $\x^\top.$

\begin{thm}\label{inv}
Let $A$ be a lower triangular $n \times n$ matrix with invertible diagonal
elements and ${\cal M}$ be the unique perfect
matching in $G(A)$ consisting of the edges from $R_i$ to $C_i$,
$i = 1, \ldots, n.$ Then the entries of $B = A^{-1}$, for $1 \le j \le i \le n$,  are given by
\begin{equation}\label{bij1}
b_{ij} = \sum_{P\in{\cal P}_{ij}} \epsilon (P) w(P),
\end{equation}
where ${\cal P}_{ij}$ is the set of alternating paths from $C_i$ to $R_j$
in $G(A).$
\end{thm}

\begin{proof}{ We prove the result by induction on $n,$ the cases $n =1,2$
being easy. Assume the result for matrices of order less than $n.$
Partition $A$ and $B$ as
$$A = \left(
  \begin{array}{cc}
    A_{11} & \0^\top \\
    \x & a_{nn} \\
  \end{array}
\right), \,\, B = \left(
  \begin{array}{cc}
    B_{11} & \0^\top \\
    \y & b_{nn} \\
  \end{array}
\right).$$
Note that $b_{nn} = a_{nn}^{-1}$ and $B_{11} = A_{11}^{-1}.$

By the induction assumption, (\ref{bij1}) holds for $1 \le j \le i \le n-1.$
Thus we need to verify (\ref{bij1}) for the pairs $(n,1), \ldots, (n,n-1).$

From $BA = I$ we see that $\y A_{11} + b_{nn}\x = \0$ and hence
$\y = -a_{nn}^{-1}\x A_{11}^{-1}.$ Therefore
\begin{equation} \label{eq2}
y_j = -a_{nn}^{-1}\sum_{i=1}^{n-1} x_i b_{ij}, ~~ j = 1, \ldots, n-1.
\end{equation}

Consider any alternating path from $C_n$ to $R_j$ in $G(A).$ Any such
path must be composed of the edge from $C_n$ to $R_n,$
followed by an edge from $R_n$ to $C_i$ for some $i \in
\{1, \ldots, n-1\},$ and then
an alternating path from $C_i$ to $R_j.$

If $P$ is an alternating path from $C_i$ to $R_j,$ then denote by $P'$
the alternating path from $C_n$ to $R_j$
obtained by concatenating the edge from $C_n$ to $R_n,$ then the edge
from $R_n$ to $C_i,$ followed by $P.$
Note that
\begin{equation} \label{eq3}
\epsilon(P') w(P') = -\epsilon (P) a_{nn}^{-1}x_i w(P).
\end{equation}

By the induction assumption, $b_{ij} = \sum \epsilon(P) w(P),$
where the summation is over all alternating paths from $C_i$ to $R_j.$
Hence it follows from (\ref{eq2}) and (\ref{eq3}) that for $j = 1, \ldots, n-1,$
$$b_{nj} = y_j =-a_{nn}^{-1}\sum_{i=1}^{n-1}x_ib_{ij}= -a_{nn}^{-1}\sum_{i=1}^{n-1}x_i\left(\sum_{P\in{\cal P}_{ij}} \epsilon (P) w(P)\right)=
\sum_{P\in{\cal P}_{nj}} \epsilon (P) w(P),$$
%where the summation is over all alternating paths from $C_n$ to $R_j,$
 completing the proof.
}\end{proof}

\vskip 1em

We note some consequences of Theorem \ref{inv}. Since the weights
are noncommutative, we may take the weights to be square matrices of a fixed order.
This leads to combinatorial formulas for inverses of block triangular matrices.
For example, the usual formula
$$\left(
    \begin{array}{cc}
      A & O \\
      C & B \\
    \end{array}
  \right)^{-1}
  = \left(
      \begin{array}{cc}
        A^{-1} & O \\
        -B^{-1}CA^{-1} & B^{-1} \\
      \end{array}
    \right)$$
    is a consequence of Theorem \ref{inv}. Another example is the identity
    $$\left(
        \begin{array}{cccc}
          A & O & O & O \\
          W & B & O & O \\
          X & O & C & O \\
          O & Y & Z & D \\
        \end{array}
      \right)^{-1}$$ $$ = \left(
        \begin{array}{cccc}
          A^{-1} & O & O & O \\
          -B^{-1}WA^{-1} & B^{-1} & O & O \\
          -C^{-1}XA^{-1} & O & C^{-1} & O \\
          D^{-1}YB^{-1}WA^{-1} + D^{-1}ZC^{-1}XA^{-1} & -D^{-1}YB^{-1} & -D^{-1}ZC^{-1} & D^{-1} \\
        \end{array}
      \right).$$

\vskip 1em

We note yet another consequence of Theorem \ref{inv}.
    Let GF(2) denote the Galois field of order $2.$ The following result easily
    follows from Theorem \ref{inv}.
\begin{cor} Let $A$ be an $n \times n$ lower triangular matrix over ${\rm GF(2)}$
     such that $a_{ii} = 1$, $i = 1, \ldots, n$; and let $B = A^{-1}.$
     Let $G(A)$ be the graph associated with $A.$ Then $b_{ij} = 1$
     if and only if there are an odd number of alternating paths from
    $C_i$ to $R_j$ in $G(A).$
\end{cor}

If $A$ is a lower triangular matrix, then
    \begin{equation} \label{eq1} \left( \begin{array}{cc} O & A \\A^\top & O \\ \end{array}\right )\end{equation}
      is the (weighted) adjacency matrix of a bipartite graph with a unique perfect matching. Conversely
the adjacency matrix of a bipartite graph with a unique perfect matching can be put in
the form (\ref{eq1}) after a relabeling of the vertices.
In view of this observation, the unweighted case of Theorem~\ref{inv} can
be seen to be equivalent  to Lemma 2.1 of  Barik, Neumann and Pati \cite{bnp}. Our proof
technique is different. In the same spirit, Theorem~\ref{inv} leads to a formula
for the inverse of the adjacency matrix of a weighted tree (see Section~4) when the tree has a perfect
matching, generalizing a well-known result from \cite{bh,pk} (see also \cite[Section 3.6]{bapat}).

\begin{rem} Let $T$ be tree with nonsingular weighted adjacency matrix $A$.
Then $A^{-1}$ is the weighted adjacency matrix of a bipartite graph. The graphs that
can occur as inverses of nonsingular trees were characterized in  \cite{np}.
   Namely, a graph $G$ is the inverse of some tree if and only if $G\in{\cal F}_k$ where ${\cal F}_k$ is the family of graphs defined recursively as follows. Set ${\cal F}_1=\{P_2\}$ and for $k\ge2$ any $G\in{\cal F}_k$ is obtained from some $H\in{\cal F}_{k-1}$ by taking any vertex $u$ of $H$ and adding two new vertices $u'$ and $v$ where $u'$ is joined to all the neighbors of $u$ and $v$ (a pendant vertex) is joined to $u'$. The characterization remains valid in the more general setting when the weights of the edges come from a ring (provided the required inverses of the weights exist).
 \end{rem}
% and vertex set $\{v_1,\ldots,v_n\}$ where $v_1$
% is a pendant vertex. Assume that $T$ possesses a perfect matching $\M$, so $A$ is nonsingular. Let $D=(d_{ij})$ be the diagonal matrix with $d_{ii}=1$ if the unique path between $v_1$ and $v_i$ contains an even number of edges in $\M^c$ and $d_{ii}=-1$ otherwise.
% Then $DA^{-1}D^{-1}$ is the weighted adjacency matrix of a bipartite graph denoted by $T^{-1}$ (see \cite{bnp}).
\section{Matrices with isomorphic inverses}

In this section we consider real matrices.
It is an interesting problem to determine the triangular matrices $A$ for which $G(A)$ is isomorphic to $G(A^{-1})$.
This problem is in close connection with the one posed by Godsil \cite{g} as described below.

Let $G$ be a bipartite graph on $2n$ vertices which has a unique perfect matching $\M$.
 Then there is a lower triangular matrix $A$ such that $G=G(A)$.
  With the additional
hypothesis that the graph $G/\M$, obtained from $G$ by contracting the edges
in $\M$, is bipartite, Godsil \cite{g}  showed that $A^{-1}$
is diagonally similar to a matrix $A^+$ whose entries are nonnegative and which
dominates $A$, that is $A^+(i,j)\ge A(i,j)$ for all $1\le i,j\le n$. In turn, $A^+$ can be
regarded as the adjacency matrix of a bipartite multigraph $G^+$ in which $G$ appears
as a subgraph. In this framework,  Godsil asked for a characterization of the graphs $G$ such that $G^+$ is isomorphic to $G$.
This was answered in \cite{sc}, by showing that $G$ and $G^+$ are isomorphic if and only if $G$ is a corona of bipartite graph. The {\em corona} of a graph
is obtained by creating a new vertex $v'$ for each vertex $v$ such that $v'$
is adjacent to $v$. The following theorem is a generalization of this result.

\begin{thm} Let $A$ be a lower triangular matrix with nonnegative entries, $\M$ being the unique matching of $G=G(A)$ and such that
$G/\M$ is bipartite. Then $A$ and $A^{-1}$ have the same zero-nonzero pattern if and only if $G$ is a corona of a bipartite graph.
\end{thm}
\begin{proof}{If $G$ is a corona, by some rearranging, we may write $A$ as
$$A=\left(
      \begin{array}{cc}
        I & O \\
        A_0 & I \\
      \end{array}
    \right),$$
for some $A_0$. Hence $$A^{-1}=\left(
      \begin{array}{cc}
        I & O \\
        -A_0 & I \\
      \end{array}
    \right),$$
proving the `if' part of the theorem.

Next, assume that  $A$ and $A^{-1}$ have the same zero-nonzero pattern. To show that $G$ is a corona, it suffices to prove that the alternating paths of $G$ are of length at most $3$. By contradiction, suppose that $G$ has an alternating path of length larger than $3$ and so it has an alternating path of \underline{length 5} between $R_j$ and $C_i$, say. Since $G/\M$ is bipartite, all the alternating paths between $R_j$ and $C_i$  must have the same length mod 4 (note that two alternating paths with different lengths mod 4 between two vertices  give rise to an odd cycle in $G/\M$).
So, by Theorem~\ref{inv}, the $(i,j)$ entry of $A^{-1}$ is nonzero. Since $A$ and $A^{-1}$ have the same zero-nonzero pattern,
the $(i,j)$ entry of $A$ is nonzero and hence $R_j$ and $C_i$ are adjacent.
This implies the existence of a triangle in $G/\M$, a contradiction.
}\end{proof}

\section{Generalized inverses and matchings}

Let $A$ be an $m\times n$ matrix with entries from a ring such that $T=G(A)$ is a tree and let $\M$ be a matching
 in $T$. When $\M$ is perfect, $A$ is nonsingular and a formula for
 $A^{-1}$ may be given in terms of alternating paths, as noted at the end of Section~2.
When $\M$ is not perfect, we still may define an  $n\times m$ matrix $B=(b_{ij})$ using the alternating paths
of $\M$ in the same fashion as when $\M$ is a perfect matching.
More precisely, if $\{R_1,\ldots,R_m\}$ and $\{C_1,\ldots,C_n\}$ are color classes of $T$, then
for $1 \le j \le i \le n$,
$$b_{ij} = \sum \epsilon (P) w(P),$$
where the summation is over all alternating paths $P$ from $C_i$ to $R_j$
in $G(A)$.  We call such a matrix the {\em path matrix of $T$ with respect to $\M$}.
We show that the path matrix turns out be an outer inverse of the adjacency matrix.

\begin{thm}\label{bab} Let $A$ be an $m\times n$ matrix such that $T=G(A)$ is a tree and let $\M_1$ and $\M_2$ be two matchings in $T$ with $\M_2\subseteq\M_1$.
 Let $B_1$ and $B_2$ be $n\times m$ path matrices of $T$ with respect to $\M_1$ and $\M_2$, respectively. Then
$$B_1AB_2=B_2AB_1=B_2.$$
\end{thm}

\begin{proof}{Let $F_1$ and $F_2$ be the induced forests by $T$ on the vertices saturated by $\M_1$ and $\M_2$, respectively. Let $A_1$ and $A_2$ be the submatrices of $A$ such that $F_1=G(A_1)$ and $F_2=G(A_2)$.
Then $\M_1$ and $\M_2$ are perfect matchings for $F_1$ and $F_2$, respectively.
Let $|\M_1|=p$ and $|\M_2|=q$.
It turns out that, with an appropriate ordering of the vertices,
$$B_1=\left(\begin{array}{cc}A_1^{-1} & O_{p\times(m-p)} \\ O_{(n-p)\times p} & O_{(n-p)\times(m-p)} \end{array}\right)~~\hbox{and}~~
B_2=\left(\begin{array}{cc}A_2^{-1} & O_{q\times(m-q)} \\ O_{(n-q)\times q} & O_{(n-q)\times(m-q)} \end{array}\right).$$
Note that $A_2^{-1}$ is also a submatrix of $A_1^{-1}$, so $B_1$ is in fact of the form
$$B_1=\left(\begin{array}{c|c|c}A_2^{-1} & O & O \\ \hline * & * & O \\ \hline O & O & O \end{array}\right).$$
Then
$$AB_1=\left(\begin{array}{cc}I_{p\times p} & O_{p\times(m-p)} \\ * & O_{(m-p)\times(m-p)} \end{array}\right).$$
It follows that
$$B_2AB_1=\left(\begin{array}{c|c|c}A_2^{-1} & O & O \\ \hline O & O & O \\ \hline O & O & O \end{array}\right)=B_2.$$
The equality $B_1AB_2=B_2$ is proved similarly.}\end{proof}

With the same proof as the theorem above, we can prove even a more general statement as follows.

\begin{thm} Let $A$ be an $m\times n$ matrix such that $T=G(A)$ is a tree and let $\M_1$ and $\M_2$ be two matchings in $T$.
 If $B_1$ and $B_2$ be $n\times m$ path matrices of $T$ with respect to $\M_1$ and $\M_2$, respectively, then
$$B_1AB_2=B_2AB_1=C,$$
where $C$ is the path matrix of $T$ with respect to $\M_1\cap\M_2$.
\end{thm}

Recall that the matrix $B$ is called a {\em $2$-inverse} (or an  {\em outer inverse}) of the matrix
$A$ if $BAB = B$ (see, for example, \cite{ben}). The next result is an immediate
consequence of Theorem~\ref{bab}.

\begin{cor} Let $A$ be a matrix such that $T=G(A)$ is a tree and let $\M$
   be a matching in $T.$ If $B$ is the path matrix of $T$ with respect to $\M,$
   then $B$ is an outer inverse of $A.$
\end{cor}

\section*{Acknowledgments}
The first author acknowledges support from the JC Bose Fellowship, Department
of Science and Technology, Government of India, and also thanks IPM, Tehran,
for hospitality during a visit when this research was carried out.
The research of the second author was in part supported by a grant from IPM (No. 91050114).


\begin{thebibliography}{MM}

\bibitem{bapat} R.B. Bapat, {\em Graphs and Matrices}, Springer, London; Hindustan Book
Agency, New Delhi, 2010.
\bibitem{bnp} S. Barik, M. Neumann, and S. Pati, On nonsingular trees and a reciprocal
eigenvalue property, {\em  Linear Multilinear Algebra} {\bf54} (2006), 453--465.

\bibitem{ben} A. Ben-Israel and T.N.E. Greville, {\em Generalized
Inverses: Theory and Applications,} Second Edition, Springer, New York, 2003.

\bibitem{bh} F. Buckly, L.L. Doty, and F. Harary,  On graphs with signed inverses, {\em Networks} {\bf18} (1998), 151--157.

\bibitem{g} C.D. Godsil, Inverses of trees, {\em Combinatorica} {\bf5} (1985), 33--39.

\bibitem{np} M. Neumann and S. Pati, On reciprocal eigenvalue property of weighted trees, {\em  Linear Algebra Appl.} (2011) doi:10.1016/j.laa.2011.09.017

\bibitem{pk} S. Pavlikova and J. Krc-Jediny, On the inverse and dual index of a tree, {\em Linear Multilinear Algebra} {\bf28} (1990), 93--109.

\bibitem{sc} R. Simion and D.-S. Cao, Solution to a problem of C.D. Godsil regarding bipartite graphs with unique perfect matching, {\em Combinatorica} {\bf9} (1989), 85--89.


\end{thebibliography}
\end{document}